\documentclass[11pt]{amsart}
\usepackage[utf8]{inputenc}
\usepackage{amsmath,a4wide}
\usepackage[short-journals, initials,msc-links]{amsrefs}
\usepackage{amsfonts,xcolor,color,soul}
\usepackage{amssymb,amscd}
\usepackage{hyperref}

\usepackage{amsthm}
\newtheorem{thm}{Theorem}[section]
\newtheorem{cor}[thm]{Corollary}
\newtheorem{lem}[thm]{Lemma}
\newtheorem{prop}[thm]{Proposition}
\theoremstyle{definition}
\newtheorem{defn}[thm]{Definition}
\newtheorem{example}[thm]{Example}

\theoremstyle{remark}

\newcommand{\Hom}{\operatorname{Hom}}
\newcommand{\To}{\rightarrow}
\newcommand{\ZZ}{\mathbb{Z}}

\newcommand{\KK}{\mathbb{K}}

\newcommand{\cL}{\mathcal{L}}
\newcommand{\cR}{\mathcal{R}}
\newcommand{\g}{\mathfrak{g}}
\newcommand{\ab}{\mathfrak{a}}

\newcommand{\h}{\mathfrak{h}}
\newcommand{\cO}{\mathcal{O}}
\newcommand{\bq}{\mathbf{q}}

\newcommand{\qas}{\cO_\bq(\KK^n)}
\newcommand{\qasd}{\cO_{\bq'}(\KK^n)}
\newcommand{\Ann}{\operatorname{Ann}}

\newcommand{\End}{\operatorname{End}}
\newcommand{\chr}{\operatorname{char}}

\newcommand{\Ker}{\operatorname{Ker}}
\newcommand{\ch}{\operatorname{char}}

\newcommand{\Loc}[1]{{#1}\operatorname{-Mod}_{loc}}
\newcommand{\Mod}[1]{{#1}\operatorname{-Mod}}
\newcommand{\Comod}[1]{\operatorname{Comod-}{#1}}

\usepackage[left=2cm,right=2cm,top=2cm,bottom=2cm]{geometry}
\author{Can Hat\.{i}po\u{g}lu}
\address{Department of Mathematics, College of Engineering and Technology, American University of the Middle East, Kuwait}
\email{osman.hatipoglu@aum.edu.kw}
\author{Christian Lomp}
\address{Department of Mathematics of the Faculty of Science and  Center of Mathematics, University of Porto, Rua Campo Alegre, 687, 4169-007 Porto, Portugal}
\email{clomp@fc.up.pt}

\title{Locally finite representations over Noetherian Hopf algebras}


\begin{document}

\begin{abstract}
We study finite dimensional representations over some Noetherian algebras over a field of characteristic zero. More precisely, we give necessary and sufficient conditions for the category of locally finite dimensional representations to be closed under taking injective hulls and extend results known for group rings and enveloping algebras to Ore extensions, Hopf crossed products, and affine Hopf algebras of low Gelfand-Kirillov dimension.
\end{abstract}
\maketitle
\setcounter{section}{0}
\section*{Introduction}


This paper is concerned with the question under which conditions the category $\Loc{A}$ of locally finite representations over a Noetherian $\KK$-algebra $A$ is closed under taking injective hulls in $\Mod{A}$\footnote{ Recall that the injective hull $E(M)$ of a non-zero module $M$ can be either defined as being a minimal injective object containing $M$ or as a {\it maximal essential extension} $E$ of $M$, where a submodule is called {\it essential} if and only if it has non-zero intersection with any non-zero submodule of $E$.}, in which case we will say that $\Loc{A}$ is \emph{essentially closed}. An example of a locally finite left $A$-module is its finite dual $A^\circ$, which is known to be a coalgebra. Furthermore, any coalgebra is an injective object in the category of its comodules (see \cite{BrzezinskiWisbauer}*{3.21}). Hence the coalgebra $A^\circ$ is an injective object in $\Comod{A^\circ} =\Loc{A}$ and so a necessary condition for $\Loc{A}$ to be  essentially closed is that the finite dual $A^\circ$ is an injective $A$-module. We will see, that the injectivity of $A^\circ$ as left $A$-module is also a sufficient condition, in case $A$ is Noetherian. 

\medskip

Our study was motivated by  various results in the literature: Let $G$ be a polycyclic-by-finite group and $A=\KK[G]$ its group ring over a field $\KK$. Then $\Loc{A}$ is essentially closed (see \cites{Donkin, musson}; S. Donkin attributes this result to K. Brown in \cite{Donkin}). Actually the interest in this question dates back to works by P.  Hall and J. E. Roseblade from the 1960s and 1970s on finitely generated soluble groups (see \cite{Hall,Roseblade}). Earlier, in \cite{Matlis}*{Proposition 3}, E. Matlis had already shown, by using the Artin-Rees property, that over a commutative Noetherian ring,  the injective hull of a simple module is Artinian. For a finite dimensional Lie algebra $\g$ over  field $\KK$ with  enveloping algebra $A=U(\g)$, $\Loc{A}$ is essentially closed if and only if $\g$ is solvable or $\ch(\KK)>0$ (see \cite{Feldvoss}). The sufficiency for solvable Lie algebras in characteristic zero had been proven by S. Donkin and R. Dahlberg \cites{Dahlberg, Donkin82}, while the necessity has been shown by J. Feldvoss. An old result of N. Jacobson \cite{Jacobson}*{Proposition 2} states  that  $U(\g)$ is always a finite module over its center for a finite dimensional Lie algebra $\g$ in positive characteristic. Hence $U(\g)$ is a Noetherian PI-algebra for which locally finite representations are essentially closed (see Theorem \ref{sufficient} below).

\medskip
	
There are instances when $\Loc{A}$ is essentially closed due to ``trivial" reasons. One extreme is when all representations over $A$  are locally finite, which is precisely the case when $A$ is a finite dimensional algebra. The other extreme is when there are no non-zero finite dimensional representations. This is the case, for example, for any infinite dimensional simple $\KK$-algebra $A$, like the Weyl algebras. In case $A$ is a (not necessarily commutative) Noetherian $\KK$-algebra such that all  irreducible representations over $A$ are finite dimensional, then $\Loc{A}$ is essentially closed if and only if every finitely generated essential extension of an irreducible representation is Artinian. The latter condition has been termed condition $(\diamond)$ in a series of papers \cites{BrownCarvalhoMatczuk,CarvalhoLompPusat,CarvalhoHatipogluLomp,CarvalhoMusson,HatipogluLomp,Musson11}. One of the main results by K. Brown {\emph{et al.}} in \cite{BrownCarvalhoMatczuk} is that for an affine commutative Noetherian $\KK$-algebra $R$ and automorphism $\alpha$, the skew polynomial ring $A=R[\theta; \alpha]$ satisfies $(\diamond)$ if and only if all irreducible left $A$-modules are finite dimensional. Hence in this case $ (\diamond)$ implies $\Loc{A}$ to be essentially closed. However, in general, these two conditions are independent. For any Weyl algebra $A$, $\Loc{A}$ is essentially closed but the second Weyl algebra over the complex numbers does not satisfy $(\diamond)$ by an old result of T. Stafford. On the other hand, the enveloping algebra $A = U(\mathfrak{sl}_2)$ satisfies $(\diamond)$ by \cite{Dahlberg89} but $\Loc{A}$ is not essentially closed as we show below. 

The structure of this paper is as follows: in Section 1, we will identify the category of locally finite representations of an algebra $A$ with the category of right comodules over the finite dual $A^\circ$. The first main result of Section 2 is  Theorem \ref{generalTheorem}, which characterizes Noetherian algebras $A$, whose  locally finite representations are closed under essential extensions by reducing our problem to finite dimensional irreducible representations and by showing that this is equivalent to the injectivity of $A^\circ$:

\begin{thm} The following statements are equivalent for a Noetherian $\KK$-algebra $A$.
	\begin{enumerate}
		\item[(a)] $\Loc{A}$ is essentially closed.
		\item[(b)] The injective hull of any finite dimensional irreducible representation over $A$ is locally finite.
		\item[(c)] $\mathrm{Loc}(E)$ is a direct summand for any injective left $A$-module $E$.
		\item[(d)] $A^\circ$ is an injective left $A$-module.
	\end{enumerate}
\end{thm}

Section 2 contains also sufficient conditions to test whether $\Loc{A}$ is essentially closed, with Theorem \ref{sufficient} being the main result of the section. 

\begin{thm}The category of locally finite representations of a Noetherian  $\KK$-algebra $A$ is essentially closed if every maximal ideal of finite codimension contains an ideal $Q$ that satisfies one of the following conditions:
	\begin{enumerate}
		\item  $Q$ has the Artin-Rees property and $A/Q$ is an affine PI-algebra. 
		\item  $Q$ is polynormal and  $\Loc{A/Q}$ is essentially closed.
	\end{enumerate}
\end{thm}

This allows us to conclude in Section 3 that locally finite representations are essentially closed over any (unmixed) Ore extensions $R[x;\sigma]$ and $R[x;\delta]$ of a commutative Noetherian algebra $R$ (see \ref{unmixed_auto}). One of the additional tools here is the Rees ring of an ideal which we examine for certain maximal ideals of an Ore extension. In the final section, we turn to Noetherian Hopf algebras and iterated Hopf Ore extensions and find in Theorem~\ref{Reduction to the trivial module} conditions under which our problem can be decided by just looking at essential extensions of the trivial representation.

\begin{thm}
Let $H$ be a Hopf algebra with invertible antipode. Then $\Loc{H}$ is essentially closed if and only if the injective hull of the trivial representation is locally finite.
\end{thm}

Throughout the paper $\KK$ denotes an algebraically closed field of characteristic $0$ and algebras are considered to be $\KK$-algebras.
The paper finishes with a look at Hopf crossed products and affine Hopf algebras of finite Gelfand-Kirillov dimension, using the classification result of Brown, Goodearl, Zhang and others (see \cite{BrownGilmartinZhang,Brown-Zhang, Goodearl-Zhang,Wang-Zhang-Zhuang,Zhou-Shen-Lu}). In particular, we prove that locally finite representations of an algebra $A$ are closed under essential extensions in any of the following cases
\begin{enumerate}
\item $A=R[x;\sigma]$ or $A=R[x;\delta]$ with $R$ commutative affine, $\sigma$ an automorphism and $\delta$ a derivation.
\item $A=R[x;\sigma,\delta]$ is an Ore extension of a commutative Noetherian Hopf algebra $R$.
\item $A=R\#_\sigma H$ is an affine Noetherian Hopf algebra with bijective antipode and a crossed product of a Hopf subalgebra $R$ and commutative Hopf algebra $H$, such that $R^+$ satisfies the strong Artin-Rees property.
\item $A$ is an affine Hopf algebra domain of Gelfand-Kirillov dimension $2$ with $\operatorname{Ext}_H^1(\KK,\KK)\neq 0$.
\end{enumerate}

\section{Locally finite representations of an algebra}

Let $A$ be a $\KK$-algebra, and $A^{*} = \Hom_{\KK}(A,\KK)$ be the space of linear forms on $A$. Denote the category of left $A$-modules by $\Mod{A}$. By a \emph{locally finite (left) representation} of $A$ we mean a left $A$-module $M$ such that $Am$ is a finite dimensional subspace of $M$ for any $m \in M$. We will denote the subcategory of locally finite representations of $A$ by $\Loc{A}$. A particular example of a locally finite representation is the finite dual of $A$, which is the coalgebra
\[A^{\circ} = \{f \in A^{*} : \Ker f \text{ contains an ideal $I$ of $A$ with  } \dim(A/I) < \infty\}\]
whose comultiplication is defined by $\Delta(f) = \sum f_{1}\otimes f_{2} \in A^{\circ} \otimes A^{\circ}$ such that $f(ab) = \sum f_{1}(a)f_{2}(b)$. The functions $f_{i}$ exist since $f$ factors through the finite dimensional algebra $A/I$ and hence the composition $f \circ \mu$ of the multiplication $\mu$ of $A/I$ with $f$ is a linear form from $A/I \otimes A/I$ to $\KK$ and hence belongs to $(A/I \otimes A/I)^{*} = (A/I)^{*} \otimes (A/I)^{*}$. The left $A$-action on $A^{\circ}$ is given by $(a\cdot f)(b) = \sum f_{1}(b)f_{2}(a)$, for all $a, b \in A$ and $f \in A^{\circ}$. Since $I\cdot f = 0$, $A \cdot f$ is finite dimensional. 

We will denote the category of right $A^{\circ}$-comodules by $\Comod{A^{\circ}}$.  Any right $A^{\circ}$-comodule $(M, \rho_{M})$ becomes also a left $A$-module with the action given by $a\cdot m = \sum_{i=1}^{n}m_{i}f_{i}(a)$ for $a\in A$, $m\in M$ and $\rho_{M}(m) = \sum_{i=1}^{n} m_{i} \otimes f_{i} \in M \otimes A^{\circ}$.  In particular $A\cdot m \subseteq \sum_{i=1}^n \KK m_i$ is finite dimensional and hence any right $A^{\circ}$-comodule is a locally finite representation of $A$. Conversely, for any locally finite representation $M \in \Loc{A}$ and element $m \in M$, one can choose a $\KK$-basis $(m_{1}, \ldots, m_{n})$ of $Am$ and elements $f_{1}, \ldots, f_{n} \in A^{*}$ such that $a\cdot m = \sum_{i=1}^{n}f_{i}(a)m_{i}$ for all $a \in A$. Since $\Ann_{A}(Am)$ is a cofinite ideal of $A$ and contained in the kernel of each $f_{i}$, these linear forms are actually elements of the finite dual of $A$, i.e. $f_{i} \in A^{\circ}$. Thus $M$ is a right $A^{\circ}$-comodule with comultiplication on $M$ given by $\rho_{M}(m) := \sum_{i=1}^{n} m_{i} \otimes f_{i} \in M\otimes A^\circ$.  Thus $\Loc{A} = \Comod{A^\circ}$ (see also \cites{BrzezinskiWisbauer, BrownGoodearl, Montgomery, Dascalescu}). 

In general, the category of right comodules over a coalgebra $C$ can be identified with the subcategory $\sigma[{_{C^*}C}]$ of $\Mod{C^*}$, which consists of those $C^*$-modules that are isomorphic to submodules of factor modules of direct sums of copies of $C$ (see \cite{BrzezinskiWisbauer}*{4.3}). Here $C^*=\Hom_\KK(C,\KK)$ is a $\KK$-algebra with the convolution product and acts on $C$ by $f\cdot c = \sum_{(c)} c_1f(c_2)$, for all $c\in C$ and $f\in C^*$. Therefore, the category of locally finite representations over $A$ is identified as
\[\Loc{A} = \Comod{A^\circ} =  \sigma[ _{(A^\circ)^*}{A^\circ}] = \sigma[ _A{A^\circ}] \subseteq \Mod{A},\]
where the last equality results from the inclusion of $A$ as a subalgebra of $(A^\circ)^\circ$. 
Note that categories of type $\sigma[_RM]$ (termed Wisbauer categories), for a left $R$-module $M$ over  a ring $R$, are known to be  Grothendieck categories and hence have enough injective modules. However, in general the injective hulls in $\sigma[_RM]$ are properly contained in the injective hulls in $\Mod{R}$.

Note that essential extensions of locally finite modules may not be locally finite over an Ore extension of a not necessarily commutative Noetherian ring. Neither need locally finite representations be closed under essential extensions over Ore extension of a Noetherian $\KK$-algebra $R$ where $\Loc{R}$ is essentially closed. The following applies in particular to the enveloping algebra of semisimple Lie algebras.

\begin{lem}\label{semisimple_algebras_are_bad}
Let $A$ be an augmented $\KK$-algebra that is a domain, such that all finite dimensional representations of $A$ are completely reducible. If $\Loc{A}$ is  essentially closed, then $A=\KK$.
\end{lem}
\begin{proof}
Let $\alpha:A\rightarrow \KK$ be any algebra homomorphism and consider ${_\alpha}\KK$, a simple representation of $A$ via $\alpha$. If $\Loc{A}$ is essentially closed, then the injective hull $E(\KK)$ of $\KK$  is  locally finite. Since by hypothesis any finitely generated submodule of $E(\KK)$ is completely reducible, we have $E(\KK)=\KK$ and hence $\KK$ is an injective left $A$-module.
As an injective $A$-module, $\KK$ is divisible, i.e. for any non-zero  $a \in A$, one has  $a \cdot \KK = \KK$. Since for any nonzero element $a\in \mathrm{Ker}(\alpha)$, $a\cdot \KK = 0$ holds, we must have $\mathrm{Ker}(\alpha) = 0$, i.e. $A=\KK$.
\end{proof}

\begin{example}[Semisimple Lie algebras]
Let $\g = \KK e \oplus \KK h \oplus \KK f$ be the semisimple Lie algebra $\mathfrak{sl}_2$ over $\KK$ with the standard relations $[e,f] = h$, $[e,h] = -2e$, $[f,h] = 2f$. Then the universal enveloping algebra $U(\g)$ is an iterated Ore extension $U(\g) = D[e;\sigma, \delta]$ and $D=\KK[f][h; id, -2f \frac{\partial}{\partial f}]$, with automorphism $\sigma$ and $\sigma$-derivation $\delta$ of $D$ given by:
\[\sigma(f)=f, \qquad \sigma(h)=h-2, \qquad \delta(f)=h, \qquad \delta(h)=0.\]
By Weyl's Theorem, finite dimensional left $U(\g)$-modules are completely reducible in case $\chr(\KK)=0$. Moreover $U(\g)$ is a Noetherian domain with algebra homomorphism $\epsilon:U(\g)\rightarrow \KK$ given by 
$\epsilon(h)=\epsilon(e)=\epsilon(f)=0$. Thus by Lemma \ref{semisimple_algebras_are_bad}, $\Loc{U(\g)}$ is not essentially closed. Note that $D\simeq U(\h)$ is isomorphic to the enveloping algebra of the 2-dimensional solvable Lie algebra $\h$. By Donkin's result \cite{Donkin82}*{Proposition 2.2.2}, $\Loc{U(\h)}$ is essentially closed, while $\Loc{U(\g)}$ is not. This shows that the property of locally finite representations being closed under essential extensions is not always preserved when passing to Ore extensions. 

Although finitely generated essential extensions of irreducible $U(\g)$-modules are Artinian, they do not need to be finite dimensional since there are non-split extensions of finite dimensional and infinite dimensional irreducible representations:  let $I=U(\g)e^2 + U(\g)fe + U(\g)(h+2)$ and $J=U(\g)e + U(\g)(h+2)$ be left ideals of $U(\g)$.  We leave it to the reader to verify that 
\[\begin{CD} 0 @>>> J/I @>>> U(\g)/ I  @>>> U(\g)/J  @>>> 0 \end{CD}\]
is a non-split short exact sequence, with $J/I\simeq \KK$ being the trivial representation and $U(\g)/J$ being an infinite dimensional representation of $U(\g)$.
\end{example}

The last example shows that for a Noetherian algebra $A$, such that  $\Loc{A}$ is essentially closed then $\mathrm{Ext}^1_A(V,W)=0$ for all finite dimensional representations $W$ and infinite dimensional representations $V$ of $A$.

\section{When locally finite representations are essentially closed}\label{when locally finite reps are essentially closed}

As in the previous section, let us denote by $A^\circ$ the finite dual of a $\KK$-algebra $A$. 
\subsection{Necessary and sufficient conditions}
For any left $A$-module $M$ we set $\mathrm{Loc}(M)$ to be the sum of all finite dimensional submodules of $M$ and define the exact functor $F:\Mod{A} \rightarrow \Loc{A}$ given by $F(M):=\mathrm{Loc}(M)$. Moreover for any $M\in \Loc{A}$ and $E\in \Mod{A}$ one has $\mathrm{Hom}_A(M,E)=\mathrm{Hom}_A(M,F(E))$. Thus if $E$ is an injective left $A$-module, then $F(E)=\mathrm{Loc}(E)$ is an injective object in $\Loc{A}$.

\begin{thm}\label{generalTheorem} The following statements are equivalent for a Noetherian $\KK$-algebra $A$.
\begin{enumerate}
\item[(a)] $\Loc{A}$ is essentially closed.
\item[(b)] the injective hull of any finite dimensional irreducible representation over $A$ is locally finite.
\item[(c)] $\mathrm{Loc}(E)$ is a direct summand for any injective left $A$-module $E$.
\item[(d)] $A^\circ$ is an injective left $A$-module.
\end{enumerate}
\end{thm}
\begin{proof}
$(a) \Rightarrow (b)$ is trivial.

$(b)\Rightarrow (a)$ Let $V$ be a locally finite left $A$-module and let $E$ be an essential extension of $V$. We first show that $V$ can be assumed to be finite dimensional. Note that $E$ is locally finite if every cyclic submodule of it is finite dimensional. If $x$ is a nonzero element of $E$, then $Ax$ is an essential extension of $V \cap Ax$. Since $A$ is Noetherian, $V \cap Ax$ is a finitely generated submodule of $V$ and hence it is finite dimensional. Thus every cyclic submodule of $E$ is an essential extension of a finite dimensional submodule of $V$, and hence it suffices to consider the case in which $V$ is finite dimensional.

Now assume that $V$ is finite dimensional and let $S$ be its (finite dimensional) socle. Since $V$ is Artinian, it has an essential socle and any essential extension $E$ of $V$ is also an essential extension of $S$. Thus $E$ can be embedded into the injective hull of $S$, which is by hypothesis locally finite.

$(a)\Leftrightarrow (d)$ Following the explanation in the introduction, the category of locally finite representations $\Loc{A}$ is equal to the category of right $A^\circ$-comodules. By \cite{BrzezinskiWisbauer}*{3.21}, $A^\circ$ is an injective object in $\Comod{A^\circ}$. If $\Loc{A}$ is closed under injective hulls in $\Mod{A}$, then $A^\circ$ is equal to its injective hull as left $A$-module. Conversely, if $A^\circ$ is injective in $\Mod{A}$ and $V \in \Loc{A}$, then since every comodule $V$ over $A^\circ$ is contained in a direct sum of copies of $A^\circ$ (see \cite{BrzezinskiWisbauer}*{9.1}) we get an embedding $V \subseteq \left(A^\circ\right)^{(\Lambda)}$ for some set $\Lambda$. As $A$ is left Noetherian, $\left(A^\circ\right)^{(\Lambda)}$ is also injective. Thus the injective hull $E(V)$ of $V$ in $\Mod{A}$ is a direct summand of $\left(A^\circ\right)^{(\Lambda)}$ and therefore also locally finite.

$(a)\Leftrightarrow (c)$. Let $E$ be any injective left $A$-module and $E(\mathrm{Loc}(E))$  the injective hull of $\mathrm{Loc}(E)$ in $\Mod{A}$. By hypothesis $E(\mathrm{Loc}(E))$  is also locally finite and belongs to $\Loc{A}$. Since $F(E)=\mathrm{Loc}(E)$ is injective in $\Loc{A}$, as seen above, and essential in $E(\mathrm{Loc}(E))$, these two modules must coincide. Hence $\mathrm{Loc}(E)=E(\mathrm{Loc}(E))$ is injective in $\Mod{A}$ and splits off in $E$ as a direct summand. Conversely, let $M$ be a locally finite left $A$-module with injective hull $E$ in $\Mod{A}$. Then $M\subseteq \mathrm{Loc}(E)$ which is injective by $(c)$ and hence equal to $E$.
\end{proof}

Theorem~\ref{generalTheorem} shows that the question of determining whether $\Loc{A}$ is essentially closed for a Noetherian $\KK$-algebra can be reduced to the study of essential extensions of finite dimensional irreducible representations. Over an algebraically closed field $\KK$ and a Noetherian $\KK$-algebra $A$ with all maximal ideals being completely prime, we can go a step further and reduce the problem to the study of essential extensions of one-dimensional representations. First we note that if  $V$ is a finite dimensional irreducible representation over a $\KK$-algebra $A$ with ring of endomorphisms  $D = \End(_A V)$, then $D$ is a finite dimensional division algebra over $\KK$ and $\End(V_D)$ is a matrix ring over $D$.  Let $P = \Ann_A(V)$ be the annihilator of $V$. Then $A/P \cong \End(V_D)$ by the Jacobson Density Theorem and $P$ is in particular a maximal ideal of $A$.

\begin{cor}\label{reduction_one_dimensional} Let $A$ be a Noetherian $\KK$-algebra over an algebraically closed field $\KK$ such that maximal ideals are completely prime. Then $\Loc{A}$ is essentially closed if and only if the injective hull of one-dimensional representations are locally finite.
\end{cor}

\begin{proof}
Let $P=\Ann_A(V)$ the annihilator of a finite dimensional irreducible representation over $A$, which as mentioned before is a maximal ideal and  by hypothesis completely prime. Thus $A/P$ is a finite dimensional division algebra over $\KK$ and since $\KK$ is algebraically closed, 
$V=A/P=\KK$  is one-dimensional. Hence the corollary follows from Theorem~\ref{generalTheorem}(b).
\end{proof}

Among the examples of algebras in which maximal ideals are completely prime are certain quantum affine spaces:

\begin{defn}\label{quantumaffinespace} The {\emph{quantum affine space}} $\mathcal{O}_{\mathbf{q}}(\KK^n)$ is the $\KK$-algebra generated by $x_1, \ldots, x_n$ and relations 
$$ x_ix_j = q_{ij} x_jx_i$$
for all $1\leq i,j \leq n$, where $q_{ij}\in \KK^\times$, $q_{ii}=1$ and $q_{ji}=q_{ij}^{-1}$ (see \cite{Goodearl-Warfield}*{p.14}). 
\end{defn}

\begin{thm}[{Goodearl-Letzter, \cite{Goodearl-Letzter}*{Theorem 2.1}}]\label{GL}
Let $\qas$ be a quantum affine space for some data $\bq=(q_{ij})$. Assume that the subgroup $\Lambda$ of $\KK^\times$ generated by the $q_{ij}$ is torsionfree. Then all of the prime ideals of $\qas$ are completely prime.
\end{thm}

\subsection{Sufficient conditions}
The second part of this section will gather some sufficient conditions for $\Loc{A}$ to be essentially closed. This is for example the case if all irreducible representations are finite dimensional and $A$ satisfies the second layer condition\footnote{We refer the reader to \cite{Goodearl-Warfield} for the definition of the second layer condition.} (see  \cite{BrownCarvalhoMatczuk}*{Theorem 3.3}).

Recall that an ideal $I$ of a ring $A$ is said to have the \emph{left Artin-Rees property} if for every finitely generated left $A$-module $M$ and every submodule $N \leq M$, there is a positive integer $n$ such that $N \cap I^{n}M \leq IN$.  
An ideal generated by normal elements has the Artin-Rees property by  \cite{McConnell-Robson}*{4.2.6}.
A sequence $x_1, \ldots, x_n$ of elements of  a ring $R$ is called a \emph{normalizing sequence} if for each $j\in \{0, \cdots, n-1\}$ the image of $x_{j+1}$ in $R/\sum_{i=1}^j Rx_i$ is normal. An ideal generated by a normalizing sequence is called \emph{polynormal} (see \cite{McConnell-Robson}*{4.1.13}). The following is an adaptation of an argument by Jategaonkar:

\begin{lem}\label{Ideals-with-normal-generators}
Let $A$ be a Noetherian $\KK$-algebra, $V$ a finite dimensional irreducible representation of $A$. Suppose that  $P=\Ann_{A}(V)$ contains a polynormal ideal $Q$. Then a finitely generated essential extension $E$ of $V$ is finite dimensional if and only if $\Ann_E(Q)$ is finite dimensional.
\end{lem}

\begin{proof} We proceed by induction on the number of elements of a normalizing sequence of generators of $Q$. Suppose $Q$ is generated by one normal element $x_{1}$ of $A$. Define a map $f: E \To E$ by $f(m) = x_{1}m$. Although this is not a left $A$-module homomorphism, it is $Z(A)$-linear, with $Z(A)$ being the center of $A$, and it maps $A$-submodules of $E$ to $A$-submodules of $E$.  Since $Q$ has the Artin-Rees property by \cite{McConnell-Robson}*{4.2.6}, there exists a natural number $n \geq 1$ such that $Q^{n} E = x_{1}^{n}V = 0$. This induces a finite filtration 
\[0 \subseteq \ker (f) = \Ann_{E}(Q) \subseteq \ker (f^{2}) \subseteq  \ldots \subseteq \ker (f^{n-1}) \subseteq \ker (f^{n}) = E\]
whose subfactors are $A/Q$-modules and $f$ induces a submodule preserving chain of embeddings
\[E /\ker(f)^{n-1} \hookrightarrow \ker(f^{n-1})/\ker(f^{n-2}) \hookrightarrow \ldots \hookrightarrow \ker(f^{2})/\ker(f) \hookrightarrow \ker(f)\]
Hence $E$ is finite dimensional if and only if $\ker(f) = \Ann_{E}(Q)$ is finite dimensional.	

Assume that the assertion holds for all Noetherian algebras and essential extensions $V \leq E$ of simple left $A$-modules such that $\Ann_{A}(V)$ contains a polynormal ideal $Q$ with a normalizing sequence of generators of less than $n$ elements. Let $V \leq E$ be an essential extension of a simple left $A$-module and assume that there exists a polynormal ideal $Q \subseteq \Ann_{A}(V)$ with normalizing sequence of generators $\{x_{1}, \ldots, x_{n}\}$. By hypothesis, the submodule $E' = \Ann_{E}(x_{1})$ is  finite dimensional if and only if $E$ is finite dimensional. Let $A' = A/Ax_{1}$ and $Q' = Q/Ax_{1}$. Then $Q' \subseteq \Ann_{A'}(V)$ is polynormal and generated by the normalizing sequence $\{\bar{x}_{2},\ldots, \bar{x}_{n}\}$, where $\bar{x}_{i}$ denotes the image of the element $x_{i}$ in $A'$. Furthermore, $V \leq E'$ is an essential extension of $A'$-modules such that $Q'V = 0$. By induction hypothesis, $E'$ is finite dimensional if and only if $\Ann_{E'}(Q') = \Ann_{E}(Q)$ is finite dimensional, which proves the lemma.
\end{proof}

\begin{prop}\label{Ideals-with-normalizing-sequence-of-generators}
Let $A$ be a Noetherian $\KK$-algebra, $V$ a finite dimensional irreducible representation of $A$ and $V \leq E$ a finitely generated essential extension of $V$. Then $E$ is finite dimensional if $\Ann_{A}(V)$ contains an ideal $Q$ satisfying one of the following conditions:
\begin{enumerate}
\item $Q$ has the Artin-Rees property, such that $A/Q$ is an affine PI-algebra 
\item  $Q$ is polynormal and $\Loc{A/Q}$ is essentially closed.
\end{enumerate}
\end{prop}

\begin{proof}
(1)  Since $Q$ has the Artin-Rees property, $Q^{n}E = 0$ for some $n \geq 1$ and $E$ is a finitely generated essential extension of $V$ over the affine PI-algebra $\overline{A}:=A/Q^{n}$. By \cite{Jategaonkar}*{Theorem 2}, the injective hull of any irreducible representation over $\overline{A}$ has finite length and so does $E$. Since any irreducible representation over an affine PI-algebra is finite dimensional (see  \cite{McConnell-Robson}*{13.10.3}), $E$ is finite dimensional.

(2)  Since  $\Loc{A/Q}$ is essentially closed and $\Ann_{E}(Q)$ is a (finitely generated) essential extension of $V$,  $Ann_{E}(Q)$ is finite dimensional. By Lemma \ref{Ideals-with-normal-generators}, $E$ is finite dimensional.
\end{proof}

From Proposition \ref{Ideals-with-normalizing-sequence-of-generators} it follows obviously that $\Loc{A}$ is essentially closed for any Noetherian affine PI-algebra, which in particular applies to any Noetherian affine $\KK$-algebra that is finitely generated over a commutative affine  PI-algebra. Moreover, a direct application of the proposition is the following theorem.

\begin{thm}\label{sufficient} The category of locally finite representations of a Noetherian  $\KK$-algebra $A$ is essentially closed if every maximal ideal of finite codimension contains an ideal $Q$ that satisfies one of the following conditions:
\begin{enumerate}
\item  $Q$ has the Artin-Rees property and $A/Q$ is an affine PI-algebra. 
\item  $Q$ is polynormal and  $\Loc{A/Q}$ is essentially closed.
\end{enumerate}
\end{thm}

Before we apply Theorem \ref{sufficient} to quantum affine spaces, we first turn our attention to finite normalizing ring extensions. Recall that a ring extension $R\subseteq S$ is said to be \emph{finite normalizing}, if there exists a finite set of elements $\{x_1,\ldots, x_n\}$ of $S$ that generates $S$ as $R$-module, such that $Rx_i=x_iR$ for all $i$.
 Let $R \subseteq S$ be any ring extension and $M$ a left $R$-module $M$. Then $\Hom_{R}(S, M)$ can be made into a left $S$-module by the action 
\[(s \cdot f)(x) = f(xs)\]
for $s, x \in S$ and $f \in \Hom_{R}(S, M)$. If moreover, $M$ is a left $S$-module, then we can define the map
\[\varphi: M \To \Hom_{R}(S,M), \quad m \mapsto f_{m}\]
where $f_{m}(s) = sm$. This is actually an $S$-homomorphism since for any $s, x \in S$ and $m \in M$, we have
\[\varphi(sm)(x) = f_{sm}(x) = xsm = f_{m}(xs) = (s\cdot f_{m})(x) = (s\cdot \varphi(m))(x)\]
Moreover, $\varphi$ is a monomorphism because if $f_{m} = 0$ for some $m$, then $f_{m}(1) = m = 0$. 

\begin{lem}\label{finitenormalizing}
Let $R \subseteq S$ be a finite normalizing extension of $\KK$-algebras. If $\Loc{R}$ is essentially closed, then so is $\Loc{S}$.
\end{lem}

\begin{proof}
Let $M$ be a finite dimensional left $S$-module. We need to show that any finitely generated essential extension of $M$ is finite dimensional. 
By assumption, the $R$-injective hull $E$ of $M$ is a locally finite $R$-module and we will show that $\Hom_{R}(S, E)$ is a locally finite  $R$-module. Since $R\subseteq S$ is a finite normalizing extension, there exist elements $x_1, \ldots, x_n\in S$ such that $S=\sum_{i=1}^n Rx_i$ and $Rx_i=x_iR$, for all $i$. Define the $\KK$-linear map $\varphi: \Hom_{R}(S, E) \rightarrow E^n$ given by 
\[f \mapsto \varphi(f) := (f(x_1), \ldots, f(x_n))\]
Note that $\varphi$ is injective, because if $f(x_i)=0$, for all $i$, then as $S=\sum_{i=1}^{n}Rx_i$ also $f(S)=\sum_{i=1}^n Rf(x_i)=0$, i.e. $f=0$. Moreover, for all $f\in \Hom_{R}(S, E)$ and $r \in R$, there exist $r_1, \ldots, r_n \in R$ with $r_ix_i = x_ir$, for all $i$, and 
$$\varphi(r\cdot f) =(f(x_1r), \ldots, f(x_nr)) = (r_1f(x_1), \ldots, r_nf(x_n)) \in Rf(x_1)\oplus \cdots \oplus Rf(x_n).$$
Therefore $\varphi(R\cdot f) \subseteq Rf(x_1)\oplus \cdots \oplus Rf(x_n)$ is finite dimensional, as all  cyclic $R$-submodules $Rf(x_i)$ of $E$ are finite dimensional. Since $\varphi$ is an injective $\KK$-linear map, also $R\cdot f$ is finite dimensional. Thus, $\Hom_{R}(S, E)$ is locally finite as an $R$-module. 

It is a standard fact that $\Hom_{R}(S, E)$ is an injective left $S$-module (see for example \cite{Lam}*{(3.6B) Corollary}). Moreover, $M$ embeds into $\Hom_{R}(S,M)\subseteq \Hom_{R}(S, E)$ as a left $S$-module via $m \mapsto f_{m}$ where $f_{m}(s) = sm$. Hence, $\Hom_{R}(S, E)$ is an injective $S$-module containing $M$. Therefore, if $U$ is a finitely generated essential extension of $M$ as left $S$-module, then it is isomorphic to a submodule of $\Hom_{R}(S, E)$. Since $U$ is also finitely generated as an $R$-module and $\Hom_{R}(S, E)$ is locally finite, $U$ must be finite dimensional. This completes the proof.
\end{proof}

Denote by $ \Xi(A) := \mathrm{Alg}_{\KK}(A,\KK)$ the set of $\KK$-algebra homomorphisms $\alpha: A\rightarrow \KK$. For any $\alpha \in \Xi(A)$, we define $_\alpha \KK$ to be the one-dimensional representation with the action given by $a \cdot 1 := \alpha(a)$ for all $a \in A$. 

\begin{thm}\label{CorGL} Let $\KK$ be algebraically closed. Let $\qas$ be a quantum affine space for some data $\bq=(q_{ij})$. Then locally finite representations of $\qas$ are closed under essential extensions.
\end{thm}

\begin{proof}
	Assume first, that the group $\Lambda=\langle q_{ij}\rangle \subseteq \KK^\times$ is torsionfree. By Goodearl and Letzter's Theorem \ref{GL}, every prime ideal of $\qas$ is completely prime. By Corollary \ref{reduction_one_dimensional}, $\Loc{\qas}$ is essentially closed if and only if essential extension of the one-dimensional modules $_{\chi}\KK$ over $\qas$ are locally finite, for all $\chi \in \Xi(\qas)$. Let $\chi\in \Xi(\qas)$ and $M=\mathrm{Ker}(\chi)$. The defining relations of the quantum affine space, $x_ix_j=q_{ij}x_jx_i$, yield that \begin{equation}\label{eq_qas} (1-q_{ij})\chi(x_i)\chi(x_j) = 0, \qquad \forall 1\leq i,j \leq n.\end{equation} Let $I=\{ i : \chi(x_i)=0\}$. Then for any $i,j \not\in I$, equation (\ref{eq_qas}) yields $q_{ij}=1$. 
 Let $P=\langle \{ x_i :  i\in I \}\rangle \subseteq M$ be the ideal generated by those $x_i$ with $\chi(x_i)=0$. Then $P$ is generated by normal elements and ${\mathcal{O}_{\mathbf{q}}(\KK^n)}/P \simeq \KK[\{ x_i : i\not\in I\}]$ is a commutative Noetherian domain and has the property that locally finite representations are closed under essential extensions. Hence ${\mathcal{O}_{\mathbf{q}}(\KK^n)}$ satisfies condition (2) of  Theorem \ref{sufficient} and has the property that locally finite representations are closed under essential extensions.
 
Now suppose $\Lambda =\langle q_{ij} : i,j \rangle$ is possibly not torsionfree, then as a finitely generated Abelian group, its torsion submodule $T(\Lambda)$ is finite. Let $N=|T(\Lambda)|$ be the order of its torsion group, then $\lambda\in T(\Lambda)$ if and only if $\lambda^N=1$. In particular $\{ \lambda^N : \lambda \in \Lambda\}$ is a torsionfree subgroup of $\Lambda$. For each $i,j$ set $p_{ij}:=q_{ij}^{N^2}$ and $y_i := x_i^N$. Then $y_iy_j = p_{ij}y_jy_i$, for all $i,j$, and the subring  of $\qas$, generated by $y_1, \ldots, y_n$ can be identified with $\qasd$ for $\bq'=(p_{ij})$. Since  $\Lambda' = \langle p_{ij} : i,j \rangle$ is torsionfree, the first part of this proof applies and shows that $\qasd$ is essentially closed. Moreover, $\qas$ is generated over $\qasd$ by the set of monomials $\{ x_1^{i_1}\cdots x_n^{i_n} \mid 0\leq i_1, \ldots, i_n < N^2 \}$, such that all these monomials commute with the generators $y_i$ of $\qasd$ up to a scalar. Thus,  $\qas$ is a finite normalizing extension of $\qasd$ and Lemma \ref{finitenormalizing} shows that $\qas$ is essentially closed.
\end{proof}

\begin{example}\label{quantumplane}
A particular example is the quantum plane $\mathcal{O}_{\mathbf{q}}(\KK^2)$. In this case, $\Loc{\mathcal{O}_{\mathbf{q}}(\KK^2)}$ is always essentially closed, while by  \cite{CarvalhoMusson},  $\mathcal{O}_{\mathbf{q}}(\KK^2)$ satisfies $(\diamond)$ if and only if $q$ is a root of unity.
\end{example}

\section{Locally finite representations over Ore extensions }\label{Locally finite representations over Ore extensions}

In order to study locally finite representations over Ore extensions, we will first look at conditions to deduce that certain ideals have the Artin-Rees property. A stronger condition is the property that these ideals have a Noetherian Rees ring. We will revise the necessary background first. Let $A$ be a ring and $I$ be an ideal in $A$. The \emph{Rees ring} of $I$ is the subring $\mathcal{R}_A(I)$ of the polynomial ring $A[t]$ generated by $A + It$. In other words, 
\[\mathcal{R}_A(I) = A \oplus  It \oplus I^{2}t^{2} \oplus  \ldots \oplus I^{n}t^{n} \oplus  \ldots   =  A\oplus \bigoplus_{n=1}^\infty I^n t^n \subseteq A[t]\]
and hence a typical element $f$ of $\mathcal{R}_A(I)$ has the form $f = \sum a_{k}t^{k}$ where $a_{k} \in I^{k}$ for each $k\geq 1$.

The reason why we are interested in Rees rings is that an ideal with a Noetherian Rees ring satisfies the Artin-Rees property (see \cite{Goodearl-Warfield}*{Lemma 13.2}).\label{Ideals with Noetherian Rees rings have AR property}
We will say that an ideal $I$ of $A$ has the  \emph{strong Artin-Rees property} if $\mathcal{R}_A(I)$ is Noetherian. 
The following result is basically  \cite{Goodearl92}*{Lemma 5} (see also \cite{Bell}*{Lemma 7.1}). Recall that an ideal $I$ of a ring $R$ with automorphism $\sigma$ and $\sigma$-derivation $\delta$ is called $(\sigma,\delta)$-invariant if $\sigma(I)\subseteq I$ and $\delta(I)\subseteq I$.

\begin{lem}\label{Extensions of Ideals with Noetherian Rees Rings}
Let $R$ be a Noetherian ring, $\sigma$ be an automorphism of $R$, $\delta$ be a $\sigma$-derivation and $A= R[x;\sigma,\delta]$ be the associated Ore extension. If $I$ is an $\sigma$- and $\delta$-invariant ideal of $R$ having the strong Artin-Rees property, then $IA$ is an ideal of $A$ with the strong Artin-Rees property.
\end{lem}
\begin{proof}
Let $t$ be an indeterminate and identify $A[t]$ with $R[t][x,\sigma,\delta]$ where we extend $(\sigma,\delta)$ to $R[t]$ by $\sigma(t) = t$ and $\delta(t) = 0$. By hypothesis, all powers $I^n$ are $\sigma$- and $\delta$-invariant ideals, whence the Rees ring  $\mathcal{R}_R(I) = \bigoplus_{j=1}^{\infty} I^{j}t^{j} \subseteq R[t]$ 
is closed under $\sigma$ and $\delta$. In particular, $\sigma$ restricts to an automorphism of $\mathcal{R}_R(I)$. Moreover, $IA = AI$ and so $(IA)^{j} = I^{j}A$ for all $j \geq 0$. It follows that $\mathcal{R}_R(I)[x;\sigma,\delta] = \mathcal{R}_A(IA)$. By assumption, $\mathcal{R}_R(I)$ is Noetherian and, as an Ore extension, so is $\mathcal{R}_A(IA)$.
\end{proof}

Lemma \ref{Extensions of Ideals with Noetherian Rees Rings} and Theorem \ref{sufficient} allow us to conclude that locally finite  representations over unmixed Ore extensions of commutative Noetherian $\KK$-algebras are essentially closed as we will show in the next theorem.

\begin{thm}\label{unmixed_auto}\label{unmixed}
Let $R$ be a commutative Noetherian $\KK$-algebra with a $\KK$-algebra automorphism $\sigma$ and a $\KK$-linear derivation $\delta$.
Then locally finite representations over $R[x;\sigma]$ as well as locally finite representations over $R[x;\delta]$ are essentially closed.
\end{thm}

\begin{proof} 
Let $S=R[x;\sigma]$ and $V$ be a finite dimensional simple left $S$-module, $P=\Ann_S(V)$ be its annihilator and $E$ any finitely generated essential extension of $V$.  If  $x \in P$, then $Q= Sx$ is an ideal contained in $P$ that is generated by the normal element $x$. Since $S/Q \simeq R$ is commutative Noetherian, $\Loc{S/Q}$ is essentially closed and by Lemma \ref{Ideals-with-normalizing-sequence-of-generators}(2), $E$ is finite dimensional.

If $x \not\in P$, then  $P \cap R$ is an $\sigma$-prime ideal of $R$ \cite{McConnell-Robson}*{10.6.4(ii)}. Moreover,  $P\cap R$ has a Noetherian Rees ring, since it is an ideal of the commutative Noetherian ring $R$. By Lemma \ref{Extensions of Ideals with Noetherian Rees Rings}, $Q=(P\cap R)S \subseteq P$ is an ideal contained in $P$ with the Artin-Rees property. If  $P =Q$, then $P$ has the Artin-Rees property and by Lemma \ref{Ideals-with-normalizing-sequence-of-generators}(2), $E$ is finite dimensional. If $P\neq Q$, then by \cite{Irving}*{Theorem 4.3} $\sigma$ induces a finite order automorphism $\overline{\sigma}$ on $R/ (P\cap R)$ and by \cite{Damiano-Shapiro}*{Corollary 10} it follows that $S/Q \cong (R/(P\cap R))[x; \overline{\sigma}]$ is a PI ring. In this case it follows from Theorem~\ref{sufficient} that $E$ is finite dimensional.

Now let $S=R[x;\delta]$. By Sigurdsson's Theorem \cite{Goodearl-Warfield}*{Theorem 10.23}, every prime ideal of $S=R[x;\delta]$ is completely prime. Therefore, by Corollary~\ref{reduction_one_dimensional} it is enough to consider one-dimensional simple modules $V={_{\chi}\KK}$, for some character $\chi \in \Xi(S)$.  Let $P=\mathrm{Ker}(\chi) = \mathrm{Ann}_S(V)$. Then $P\cap R$ is $\delta$-stable, since for any $a \in P\cap R$, we have $\delta(a) = xa - ax  \in P\cap R$. Now $P \cap R$ is a $\delta$-invariant character ideal with the strong Artin-Rees property, hence by Lemma~\ref{Extensions of Ideals with Noetherian Rees Rings} $Q:= (P\cap R)S$ has the strong Artin-Rees property too. Moreover, by \cite{Brown2015}*{Theorem 4.2}, $Q$ is an ideal of $S$ such that $S/Q \cong \mathbb{K}[x]$, i.e. a PI-algebra. It follows from Theorem~\ref{sufficient} that finitely generated essential extensions of $V$ are finite dimensional.
\end{proof}

For the simplest mixed case, a general Ore extension over $\KK[x]$, we can also show that locally finite representations are essential closed.

\begin{cor}
Let $\KK$ be an algebraically closed field of characteristic zero. Then essential extensions of locally finite left $\KK[x][y;\sigma, \delta]$-modules are locally finite. 
\end{cor}
\begin{proof}
The automorphisms of $\KK[x]$ have the form $\sigma(x) = qx + b$ where $q$ and $b$ are some scalars. 
By \cite{AlevDumas}, Ore extensions $A = \KK[x][y;\sigma, \delta]$ of the ring $\KK[x]$ are isomorphic to either $\KK[x,y]$, the quantum plane $\cO_\bq(\KK^2)$, the quantized Weyl algebra $A_{1}(\KK,q)$, or a differential operator ring $\KK[x][y;\delta]$. 
We have seen already, that locally finite representations are essentially closed for $\KK[x,y]$ (Theorem~\ref{sufficient}), $\cO_\bq(\KK^2)$ (Example~\ref{quantumplane}) and  $\KK[x][y;\delta]$ (Theorem~\ref{unmixed}).

Let $A = A_{1}(k,q) = k[x][y;\sigma, \delta]$ be the quantized Weyl algebra, $q \neq 1$ and $q$ is not a root of unity. Then the element $u = yx - xy$ is a normal element in $A$, every nonzero prime ideal of $A$ contains $u$ by \cite{Goodearl92}*{Theorem 8.4(a)} and $A/uA \cong \KK[x,x^{-1}]$ (see \cite{Goodearl92}*{Proposition 8.2}). Since $A/uA$ is commutative Noetherian, $\Loc{A/uA}$ is essentially closed and again by Theorem~\ref{sufficient}, $A$ has the same property. When $q$ is a root of unity, $A$ becomes a finitely generated module over a central subring, hence a PI-algebra and Theorem~\ref{sufficient} applies again.
\end{proof}

For an arbitrary commutative coefficient ring $R$, we do not have a definite answer when the Ore extension $S = R[x;\sigma, \delta]$ is essentially closed. However, in the case where the extension $R \subseteq S$ is a Hopf-Ore extension, we can get a positive answer. Recall from \cite{Brown2015}*{Definition 2.1} that an Ore extension $S = R[x; \sigma, \delta]$ of a Hopf algebra $R$ over $\KK$ is a Hopf-Ore extension of $R$ if $S$ is a Hopf algebra over $\KK$ with Hopf subalgebra $R$ such that there are $a, b \in R$ and $v, w \in R \otimes R$ such that $\Delta(x) = a \otimes x + x \otimes b + v(x \otimes x) + w$.

\begin{thm}
Let $R \subseteq A= R[x;\sigma,\delta]$ be a Hopf-Ore extension where $R$ is a commutative Noetherian Hopf algebra domain. Then $A$ is essentially closed.
\end{thm}
\begin{proof}
By Molnar's Theorem \cite{Molnar}, $R$ is an affine $\KK$-algebra. According to \cite{Brown2015}*{Theorems 5.1 \& 5.2}, there is always a change of variables so that the Ore extension is either a differential type or an automorphism type. Since the coefficient ring is commutative, it follows from Theorem~\ref{unmixed_auto}  that $A$ is essentially closed.
\end{proof}

\section{Locally finite representations over some Hopf algebras}
%
In this  last section, we restrict ourselves mostly to Hopf algebras $H$. We have seen in Section~\ref{when locally finite reps are essentially closed} that it is sometimes possible to reduce the study of locally finite representations to the study of one-dimensional representations. In particular we will prove conditions under which $\Loc{H}$ is essentially closed provided the injective hull of the trivial representation is locally finite. This will be always the case when the antipode $S$ of $H$ is invertible. This is the case for Noetherian semiprime Hopf algebras as shown by Skryabin in \cite{Skryabin}*{Corollary 1}.

\subsection{Reduction to the trivial representation}
The category of left $H$-modules is a tensor category. In particular, for any left $H$-modules $U$ and $V$, the tensor product $U \otimes V$ is a left $H$-module with the left $H$-action defined as $h \cdot (u \otimes v) = \sum (h_{1}\cdot u)\otimes (h_{2}\cdot v)$, for all $h \in H, u \in U$ and $v \in V$. Also, $\Hom_{\KK}(U ,V)$ is a left $H$-module with left $H$-action defined as $(h \cdot f)(u) = \sum h_{1}\cdot f(S(h_{2}) \cdot u),$ for all $h \in H, u \in U$ and $f \in \Hom_{\KK}(U, V)$. In case the antipode $S$ is invertible, we define another left $H$-module structure on the dual space $V^*$ of $V$  by $(h \cdot \varphi)(v) = \varphi(S^{-1}(h)\cdot v),$ for all $h \in H, v \in V$ and $\varphi \in V^{*}$

\begin{lem}\label{isoHopf}
If $H$ is a Hopf algebra with invertible antiopode, $U$ and $E$ are left $H$-modules and $V$ is a finite dimensional representation, then
$\Hom_H(U,E\otimes V) \simeq \Hom_H(U\otimes V^*, E)$ as left $H$-modules.
\end{lem}

\begin{proof}
Define the map $\Phi: \Hom_{\KK}(U, E \otimes V) \To \Hom_{\KK}(U \otimes V^{*}, E)$ by the rule 
\[\Phi(f)(u \otimes \varphi) = (\mathrm{id}_E \otimes \varphi)(f(u)),\]
for all $f \in \Hom_{\KK}(U, E \otimes V), \varphi \in V^{*}$ and $u \in U$. It is clear that $\Phi$ is an isomorphism of $\KK$-vector spaces with inverse given by  $\Phi^{-1}: \Hom_{\KK}(U\otimes V^*, E)\rightarrow \Hom_{\KK}(U,E\otimes V)$ and 
\[\Phi^{-1}(f)(u)=\sum_{i=1}^n f(u\otimes p_i)\otimes b_i,\]
where  $\{b_1,\ldots, b_n\}$ is a basis of $V$, $\{p_1,\ldots, p_n\} \subseteq V^*$  a    dual basis satisfying $p_i(b_j)=\delta_{ij}$, for all $i, j$. 
%

We will prove that $\Phi$ is $H$-linear. For all $h \in H$, we denote the action of $h$ on $U$ by $\lambda_{h}^{U}$, i.e.  $\lambda_{h}^{U}(u) = h\cdot u$. Note that for any $\varphi \in V^{*}, h \in H$ and $e \otimes v \in E \otimes V$ we have 
\begin{equation}\label{eq1}
(\mathrm{id}_E\otimes \varphi)(h\cdot (e\otimes v)) = \sum h_1\cdot e  \varphi(h_2\cdot v)
= \sum h_1 \cdot \left( \mathrm{id}_E \otimes \varphi \lambda^V_{h_2} \right) (e\otimes v).
\end{equation}
%
%
%
Note also that $\left(\varphi\lambda_h^V\right)(v)=\varphi(h\cdot v) = (S(h)\cdot \varphi)(v)$, for all $v\in V$ and $h\in H$. Hence,
\begin{equation}\label{eq3}
 \varphi\lambda_{h}^{V} = S(h)\cdot \varphi.
\end{equation}
Therefore:
\begin{eqnarray*}
\Phi(h\cdot f)(u\otimes \varphi) &=& \sum (\mathrm{id}_E\otimes \varphi)(h_1\cdot f(S(h_2)\cdot u))\\
&=& \sum h_1 \cdot (\mathrm{id}_E \otimes \varphi\lambda_{h_2}^V)(f(S(h_3)\cdot u))\qquad \mbox{ by (\ref{eq1})}\\
&=& \sum h_1 \cdot (\mathrm{id}_E \otimes S(h_2)\cdot \varphi)(f(S(h_3)\cdot u))\qquad \mbox{ by (\ref{eq3})}\\
&=& \sum h_1 \cdot \Phi(f)( (S(h_3)\cdot u) \otimes (S(h_2)\cdot \varphi) )\\
&=& \sum h_1 \cdot \Phi(f)( S(h_2)\cdot (u\otimes \varphi))\\
&=& (h\cdot \Phi(f))(u\otimes \varphi).
\end{eqnarray*}
This shows  $\Phi(h\cdot f) = h\cdot \Phi(f)$. Thus  $\Phi$ is left $H$-linear and hence an isomorphism of left $H$-modules. The \emph{$H$-invariants} of a left $H$-module $V$ are defined as  $V^H = \{v\in V: h\cdot v = \epsilon(h)v, \forall h\in H\}$. Given two left $H$-modules $U$ and $W$, it is not difficult to see, that $\Hom_{\KK}(U,W)^H = \Hom_H(U,W)$ consists of the $H$-linear functions from $U$ to $W$.
Since $H$-invariant elements are preserved under $H$-linear isomorphisms, we conclude that 
\[\Hom_{H}(U, E \otimes V) = \Hom_{\KK}(U, E \otimes V)^H \underset{\Phi}{\simeq} \Hom_{\KK}(U \otimes V^{*}, E)^H = \Hom_{H}(U \otimes V^{*}, E).\]
\end{proof}

\begin{thm}\label{Reduction to the trivial module}
Let $H$ be a Hopf algebra. If $H$ has a bijective antipode  or if all finite dimensional representations of $H$ are $1$-dimensional, then $\Loc{H}$ is essentially closed if and only if the injective hull of $E(\KK)$ of the trivial representation is locally finite.
\end{thm}

\begin{proof}
Suppose $S^{-1}$ exists. Let $V$ be finite dimensional representation and $E$ an injective left $H$-module. By the isomorphism $\Hom_{H}(U, E\otimes V) \simeq \Hom_{H}(U \otimes V^{*}, E)$ of Lemma \ref{isoHopf},  $E \otimes V$ is injective in $\Mod{H}$, because $\Hom_{H}(-, E\otimes V)$ is the composition of the exact functors $-\otimes_{\KK} V^*$ and $\Hom_H(-, E)$. Choosing $E = E(\KK)$ as the injective hull of the trivial $H$-module $\KK$, we have that $V \simeq \KK \otimes V \subseteq E \otimes V$ and therefore $E(V) \subseteq E \otimes V$. Hence if $E(\KK)$ is locally finite, then also all injective hulls of finite dimensional representations are.

Suppose all finite dimensional representations of $H$ are $1$-dimensional. For any character $\chi\in \Xi(H)$ let $_\chi\KK := \KK v$ denote the $1$-dimensional irreducible representation of $H$  with action given by  $a\cdot v := \chi(a)v$, for all $a\in H$.  Any $1$-dimensional irreducible representation of $H$ is of that form and the trivial representation in particular is $_\epsilon\KK$. For a character $\chi \in \Xi(H)$ we define the right winding automorphism (see \cite{BrownGoodearl}*{I.9.25}):
$$ \sigma := \tau^r_\chi : H\rightarrow H, \qquad \mbox{ given by } \qquad \sigma(a) = \sum_{(a)} a_1\chi(a_2).$$
The inverse of $\sigma$ is given by $\sigma^{-1}=\tau^r_{\chi \circ S}$.
Note that $\epsilon \circ\sigma = \chi$, i.e. $\chi\circ\sigma^{-1}=\epsilon$, because
\[\epsilon(\sigma(a)) = \sum_{(a)} \epsilon(a_1)\chi(a_2) = \chi\left(\sum_{(a)} \epsilon(a_1)a_2\right) = \chi(a), \qquad \forall a\in H\]

Given any automorphism $\sigma$ of $H$ and left $H$-module $M$ we denote by $_\sigma M$ the left $H$-module with $H$-action given by 
$a.m := \sigma(a)m$, for all $a\in H$ and $m\in M$. If $N \subseteq M$ is any essential extension of left $H$-modules, 
then ${_\sigma N} \subseteq {_\sigma M} $ is also an essential extension of left $H$-modules, because for any $0\neq m\in M$ there exists $a\in H$ with $0\neq am \in N$. Thus $ \sigma^{-1}(a) \cdot m = am\in N\setminus\{0\}$.

In particular, for a character $\chi \in \Xi(H)$ and the winding automorphism $\sigma=\tau^r_\chi$ we have that the $H$-action on $_{\sigma^{-1}}(_\chi\KK)$ is equal to the trivial action, because
\[ a.v = \sigma^{-1}(a)\cdot v = \chi(\sigma^{-1}(a))v = \epsilon(a)v, \qquad \forall a\in H.\]
Thus, if ${_\chi\KK} \subseteq M$ is an essential extension of left $H$-modules, then 
${_\epsilon\KK} = {_{\sigma^{-1}}(_\chi\KK)} \subseteq {_{\sigma^{-1}}M}$ is an essential extension of left $H$-modules and if furthermore any finitely generated essential extension $M$ of $_\epsilon\KK$ is finite dimensional, then ${_{\sigma^{-1}}M}$ must be finite dimensional and so must be $M$. 
\end{proof}

\subsubsection{} A large class of algebras whose finite dimensional representations are $1$-dimensional has been found by Goodearl and Letzter:

\begin{thm}[{Goodearl-Letzter, \cite{Goodearl-Letzter}*{Theorem 2.3}}]\label{GL_SIG}
Let $A=\KK[x_1][x_2; \sigma_2, \delta_2] \cdots [ x_n; \sigma_n, \delta_n]$ be an iterated Ore extension over $\KK$.
For $1\leq i \leq n$ set $A_i=\KK[x_1][x_2; \sigma_2, \delta_2] \cdots [ x_i; \sigma_i, \delta_i]$ and assume that the following conditions hold:
\begin{enumerate}
\item[(a)] $\sigma_i$ is a $\KK$-algebra automorphism of $A_{i-1}$, and $\delta_i$ is a $\KK$-linear $\sigma_i$-derivation of $A_{i-1}$.
\item[(b)] For each $j=1,\ldots, i-1$ there exists $\lambda_{ij}\in \KK^\times$ such that $\sigma_i(x_j)=\lambda_{ij}x_j$.
\item[(c)] There exists $q_i \in \KK^\times$ such that $\delta_i\sigma_i(x_j)=q_o\sigma_i\delta_i(x_j)$ for $j=1,\ldots, i-1$.
\item[(d)] Either $q_i$ is not a root of unity or $q_i=1$ and the $\chr(\KK)=0$.
\end{enumerate}
Further assume that the subgroup $\Lambda$ of $\KK^\times$ generated by the $\lambda_{ij}$ is torsionfree. Then all prime ideals of $A$ are completely prime.
\end{thm}

In particular the only finite dimensional representations are the $1$-dimensional ones, under our running assumptions of $\KK$ being algebraically closed and having characteristic zero. Examples of algebras that can be presented as in Goodearl and Letzter's Theorem are certain quantum algebras over $\KK$, like the  multiparameter coordinate ring of {\it quantum $n\times n$ matrices},  $\cO_{\lambda, \bq}(M_{n}(\KK))$, the multiparameter {\it quantized Weyl algebras of degree $n$},  $A_n^{Q,\Gamma}(\KK)$, the coordinate ring of {\it quantum symplectic $2n$-space}, $\cO_q(\mathfrak{sp} \KK^{2n})$ and the coordinate ring of {\it quantum Euclidean $n$-space}, $\cO_q(\mathfrak{o} \KK^n)$ (see the survey \cite{Goodearl_survey} for definitions).

\subsubsection{}
Skryabin showed in \cite{Skryabin}*{Corollary 1} that a Hopf algebra $H$ has a bijective antipode, provided it  can be embedded into a left perfect ring $Q$ such that $Q$ is an essential extension of $H$ as right $H$-module.  Hence Theorem \ref{Reduction to the trivial module} applies in particular to semiprime Noetherian Hopf algebras.

\subsubsection{}
  Recall that a graded Hopf algebra is a Hopf algebra $H$ equipped with a grading $H=\bigoplus_{n\geq 0} H_n$ such that $H$ is simultaneously a graded algebra and a graded coalgebra, and the antipode preserves the given grading. Such a graded Hopf algebra is called connected if $H_0$ is one-dimensional. Recently, Zhou {\emph{et al.}} showed in \cite{Zhou-Shen-Lu}*{Theorem B} that any connected graded Hopf algebra $H$ of finite Gelfand-Kirillov dimension over a field $\KK$ of characteristic zero, can be represented as an iterated Hopf Ore extension (IHOE) of the form $H = \KK[x_{1}][x_{2};\delta_{2}] \ldots [x_{n}; \delta_{n}]$.  If for all $1\leq i<j \leq n$ one had  $\delta_j(x_i)\in \KK[x_{1}][x_{2};\delta_{2}] \ldots [x_{i}; \delta_{i}]^+$, then by induction and the following Lemma one could conclude that the augmentation ideal $H^+$ of $H$ is polynormal and hence $H$ is essentially closed.


\begin{lem}\label{polyLemma} Let $R$ be a ring, $\delta$ a derivation of $R$ and $A=R[y;\delta]$. If $I$ is a polynormal ideal of $R$ with a  normalizing sequence $(x_1,\ldots, x_n)$ such that $ \delta(x_i) \in Rx_1 + \cdots + Rx_i$ holds, for all $1\leq i \leq n$. Then $AI$ is polynormal.
\end{lem}
\begin{proof}
We prove this lemma by induction on the number of generators of $I$. For $n=1$, $I=Rx=xR$ is generated by a normal element, which satisfies $\delta(x) = fx = xg$ for some $f,g\in R$. Then $yx =xy+\delta(x) = x(y+g)$ and $xy=yx-\delta(x) = (y - f)x$ shows $Ax=xA$, i.e. $AI=Ax$ is polynormal.

Let $n\geq 1$ and suppose we have proven the statement of the lemma for all rings $R$ with derivations $\delta$ and polynormal ideals $I$ generated by a normalizing sequence $(x_1,\ldots, x_n)$ that satisfies  $ \delta(x_i) \in \sum_{j=1}^i Rx_j$, for all $1\leq i \leq n$. Suppose that $I$ is a polynormal ideal generated by a normalizing sequence $(x_1,\ldots, x_n, x_{n+1})$ that satisfies $ \delta(x_i) \in \sum_{j=1}^i Rx_j$, for all $1\leq i \leq n+1$. Let $A=R[y;\delta]$. Since $x_1$ is normal in $R$, we know that $x_1$ is normal in $A$ and hence $Ax_1$ is an ideal. Let $\overline{R}=R/Rx_1$ and $\overline{A} = A/Ax_1$. Then $\delta$ extends to a derivation $\overline{\delta}$ of $\overline{R}$ such that $\overline{R}[\overline{y};\overline{\delta}] \simeq \overline{A}$. The ideal $\overline{I}=I/Rx_1$ is polynormal, generated by the normalizing sequence $(\overline{x_2}, \ldots, \overline{x_{n+1}})$ that satisfies $\overline{\delta}(\overline{x_i}) \in \sum_{j=1}^i \overline{R}_j \overline{x_j}$, where $\overline{x_i}=x_i+Rx_1$. By induction $\overline{A}\overline{I} = AI/Ax_1$ is a polynormal ideal of $\overline{A}$ and therefore $AI$ is a polynormal ideal of $A$.
\end{proof}

Consider an iterated Hopf-Ore extension  of the  form $H = \KK[x_{1}][x_{2};\delta_{2}] \ldots [x_{n}; \delta_{n}]$ as  in the case of  \cite{Zhou-Shen-Lu}*{Theorem B} for a connected graded Hopf algebra $H$ of finite Gelfand-Kirillov dimension. Set $H_{(i)}:=\KK[x_{1}][x_{2};\delta_{2}] \ldots [x_{i}; \delta_{i}]$ for all $1\leq i \leq n$.
Note that the augmentation ideal $H_{(i)}^+ = \Ker(\epsilon)\cap H_{(i)}$ is generated by $x_1, \ldots, x_i$ as left ideal of $H_{(i)}$. Moreover, $\epsilon(\delta_{i+1}(x_j)) = \epsilon(x_{i+1})\epsilon(x_j) - \epsilon(x_j)\epsilon(x_{i+1})=0$ shows 
$\delta_{i+1}(x_j) \in H_{(i)}^+ = H_{(i)} x_1 + \cdots + H_{(i)} x_i$, for all $j\leq i$. Assuming that
\begin{equation}\label{eqpoly} \delta_{i+1}(x_j)\in H_{(i)} x_1 + \cdots + H_{(i)} x_j\end{equation}
holds for all $j\leq i$, we can conclude that $H^+$ is a polynormal ideal with normalizing sequence $\{x_1,\ldots, x_n\}$, which can be proven by using Lemma \ref{polyLemma} at each step. More concretely, at the first step, the augmentation ideal of the Hopf algebra  $H_{(1)}=\KK[x_{1}]$ is generated by $x_1$ and hence polynormal. Suppose we have already verified that the augmentation ideal $H_{(i)}^+$  is polynormal with normalizing sequence $(x_1, \ldots, x_i)$, for some $1\leq i \leq n$. Then assumption (\ref{eqpoly}) implies 
$\delta_{i+1}(x_j) \in  H_{(i)} x_1 + \cdots + H_{(i)} x_j$, for all $j \leq i$. Thus we can apply  Lemma \ref{polyLemma} and conclude that $H_{(i+1)} H_{(i)}^+ = H_{(i+1)}x_1 + \ldots + H_{(i+1)}x_i$ is a polynormal ideal of $H_{(i+1)}$. The fact that  $H_{(i+1)}/H_{(i+1)}H_{(i)}^+ = \KK[x_{i+1}]$ is commutative implies that
	$$H_{(i+1)}^+ = H_{(i+1)}H_{(i)}^+  + H_{(i+1)}x_{i+1}  = H_{(i+1)}x_1 + \cdots + H_{(i+1)}x_{i+1}$$ is polynormal. In this case Proposition \ref{Ideals-with-normalizing-sequence-of-generators} applies and shows that $E(\KK)$, the injective hull of the trivial representation $\KK$, is locally finite. Since affine Noetherian Hopf algebra domains have bijective antipode by Skryabin's result \cite{Skryabin}*{Corollary 1}, Theorem~\ref{Reduction to the trivial module} applies and shows that $H$ is essentially closed.
	
Unfortunately we were not able to verify whether \cite{Zhou-Shen-Lu}*{Theorem B}  allows us to choose the generators $x_1,\ldots, x_n$ in such a way that (\ref{eqpoly}) holds.

\subsection{Locally finite representations over crossed products}

McConnell proved in 
\cite{McConnell1967}*{Theorem 4.2} that a finite dimensional Lie algebra $\g$ is nilpotent if and only if every ideal of $U(\g)$ has the Artin-Rees property. Hence in particular the augmentation ideal $P=\mathrm{Ker}(\epsilon)$ of $U(\g)$ has this property and since $U(\g)/P = \KK$, Proposition \ref{Ideals-with-normalizing-sequence-of-generators}  applies, showing that $\Loc{U(\g)}$ is essentially closed. Donkin proved in \cite{Donkin}*{Proposition 2.2.2} that essential extensions of the trivial representation of  a solvable Lie algebra $\g$ are locally finite, by using the unique maximal nilpotent ideal $\h$ of $\g$. 
By \cite{Pickel}*{Theorem 7}, the augmentation ideal $U(\h)^+ := U(\h)\cap \mathrm{Ker}(\epsilon)$ of $U(\h)$ has the strong Artin-Rees property and Donkin goes on to prove that $U(\g)U(\h)^+$ also has the strong Artin-Rees property in $U(\g)$. Since $\g/\h$ is abelian, $U(g)/U(\g)U(\h)^+$ is commutative and Proposition \ref{Ideals-with-normalizing-sequence-of-generators} applies again.
Note that  $U(\g) = U(\h) \#_\sigma U(\g/\h)$ is a crossed product. A similar argument works in the case of the group ring of a polycyclic-by-finite group.

Let $\pi:H\rightarrow \overline{H}$ be a surjective homomorphism of Hopf algebras, which splits as a coalgebra map, i.e. there exists a coalgebra map $\gamma:\overline{H}\rightarrow H$ with $\gamma\circ \pi = id$ and such that $\gamma(\overline{1})=1$. Then 
$$R=H^{co \overline{H}} = \{ h\in H \mid (id\otimes \pi)\Delta(h) = h\otimes \overline{1} \}$$
is a normal Hopf subalgebra and $H\simeq R\#_\sigma \overline{H}$ is a crossed product, i.e. $H=R\otimes \overline{H}$ as vector spaces, $\overline{H}$ acts weakly on $R$ by $\overline{h}\cdot a = \sum \gamma(\overline{h}_1)aS(\gamma(\overline{h}_2))$, for all $a\in R$ and $\overline{h}\in \overline{H}$. The multiplication in $R\#_\sigma \overline{H}$ is given by
$$ \left(a\# \overline{h}\right) \left(b\# \overline{k}\right)  = \sum a (\overline{h}_1\cdot b)\sigma(\overline{h}_2,\overline{k}_1)\# \overline{h}_3\overline{k}_2$$
and the cocycle $\sigma:\overline{H}\times \overline{H} \rightarrow R$ is given by 
$\sigma(\overline{h},\overline{k}) = \sum \gamma(\overline{h}_1)\gamma(\overline{k}_1)S(\gamma(\overline{h}_2\overline{k}_2))$(see \cites{Montgomery, BlattnerCohenMontgomery}). 

\begin{thm}
	Let $H$ be an affine Noetherian Hopf algebra over $\KK$ with bijective antipode and surjective Hopf algebra morphism $\pi:H\rightarrow \overline{H}$ onto a commutative Hopf algebra $\overline{H}$, such that $\pi$ splits as a coalgebra map. Let $R=H^{co\overline{H}}$.
	If  $R^+$ has the strong Artin-Rees property then $HR^+$ has the strong Artin-Rees property and locally finite representations of $H$ are closed under taking injective hulls.
\end{thm}

\begin{proof}
It is enough to show that $HR^+$ has the (strong) Artin-Rees property, since $HR^+ = R^+H$ is an ideal contained in the augmentation ideal of $H$ and $H/HR^+=\overline{H}$ is commutative. Hence by Proposition \ref{Ideals-with-normalizing-sequence-of-generators}, any essential extension of the trivial representation is locally finite and by Theorem \ref{Reduction to the trivial module}, $\Loc{H}$ is essentially closed.
Since $H=R\#_\sigma \overline{H}$ is a crossed product, $\overline{H}$ acts weakly on $R$. This action can be extended to $R[t]$, where $\overline{h}\cdot t = \epsilon(h)t$, for all $\overline{h}\in \overline{H}$
 and one can form the crossed product  
 $R[t]\#_\sigma \overline{H} =  (R\#_\sigma \overline{H})[t\otimes 1] = H[t]$. By hypothesis, $\cR_R(R^+)$ is Noetherian and $\overline{H}$ acts also on it. Since $HR^+ = R^+H$ we also have $$(HR^+)^n t^n = (R^+)^nt^n \otimes \overline{H}  \subseteq \cR_R(R^+)\#\overline{H},$$ showing that the Rees ring $\cR_H(HR^+)$ of $HR^+$ can be identified with the crossed product $\cR_R(R^+)\#_\sigma \overline{H}$. As $\overline{H}$ is an affine commutative Hopf algebra and hence the coordinate ring $\cO(V)$ of an affine variety $V$, it is an image of the enveloping algebra $U(\ab)$ of a finite dimensional abelian Lie algebra (e.g. a polynomial ring). Hence $\cR_H(HR^+) = \cR_R(R^+)\#_\sigma \overline{H}$ is an image of a crossed product $\cR_R(R^+)\#_\sigma U(\ab)$, which is Noetherian by \cite{McConnell-Robson}.
\end{proof}

\subsection{Hopf algebras of low Gelfand-Kirillov dimension}
In a series of papers, Brown, Goodearl, Zhang and others started a program to classify Hopf algebras of low Gelfand-Kirillov dimension. 
Brown {\emph{et al.}} show in \cite{BrownGilmartinZhang}*{Theorem 2.2} that if $H$ is a Hopf algebra of finite Gelfand-Kirillov dimension which is connected graded as an algebra, then the antipode satisfies $S^{2} = id$. Therefore, the reduction in Theorem~\ref{Reduction to the trivial module} applies and we get that $\Loc{H}$ is essentially closed if and only if $E(\KK)$ is locally finite.

\subsubsection{} From the survey \cite{Brown-Zhang}, we recall that a famous result of Small and Warfield \cite{SmallWarfield84} says that an affine prime $\KK$-algebra of Gelfand-Kirillov dimension one is finitely generated over its center. Hence locally finite representations over an affine, prime Hopf algebra of Gelfand-Kirillov dimension one are always closed under essential extensions.

\subsubsection{}
Goodearl and Zhang classified affine Hopf algebras domains $H$ over an algebraically closed field $\KK$ of characteristic zero that  have Gelfand–Kirillov dimension two and satisfy the homological condition $\operatorname{Ext}_H^1(\KK,\KK)\neq 0$.

\begin{thm}[Goodearl \& Zhang {\cite{Goodearl-Zhang}*{Theorem 0.1}}]\label{GZ}
Let $\KK$ be algebraically closed of characteristic $0$, and let $H$ be a  Hopf $\KK$-algebra domain of Gelfand-Kirillov dimension $2$ satisfying $\operatorname{Ext}_H^1(\KK,\KK)\neq 0$. Then $H$ is Noetherian if and only if $H$ is affine if and only if $H$ is isomorphic to one of the following:
\begin{enumerate}
\item[(I)] The group algebra $\KK G$, where $G\simeq \ZZ^2$  or $G\simeq \ZZ \rtimes \ZZ$.
\item[(II)] The enveloping algebra $U(\g)$, where $\g$ is a $2$-dimensional Lie algebra.
\item[(III)] The Hopf algebras $A(n,q)$ in Example \ref{exampleA}.
\item[(IV)]The Hopf algebras $B(n, p_0, \ldots , p_s, q)$ in Example \ref{exampleB}.
\item[(V)] The Hopf algebras $C(n)$ in Example \ref{exampleC}.
\end{enumerate}
\end{thm}

We go through the Hopf algebras appearing in the theorem one by one in relation to being essentially closed. 

\begin{example}\label{exampleA}[The Hopf algebras $A(n, q)$]
For a positive integer $n$, and $q \in \KK^{\times}$, let  $A = R[y; \sigma]$, where $R=\KK[x,x^{-1}]$ and $\sigma$ is given by $\sigma(x)=qx$. By \cite{Goodearl-Zhang}*{Construction 1.1}, there is a unique Hopf algebra structure on $A$ under which $x$ is grouplike and $y$ is skew primitive. More precisely
\[ \Delta(y)=y \otimes 1 + x^n \otimes y \qquad \epsilon(y)=0 \qquad \mbox{ and } \qquad S(y) = x^{-n}y.\]
Consider the quantum affine space $\cO_\bq(\KK^3)$ with $x_1x_2=x_2x_1$ and $x_1x_3 = qx_3x_1$ while $x_3x_2=q^{-1}x_2x_3$. Then $A\simeq \cO_\bq(\KK^3)/\langle x_1x_2-1, x_2x_1-1\rangle$. 
If $q$ is not a root of unity, then $\langle 1, q , q^{-1}\rangle = \langle q \rangle \subseteq \KK^\times$ is torsionfree and by Corollary  \ref{CorGL}, locally finite representations over $\cO_\bq(\KK^3)$ are essentially closed. The same is true for the quotient algebra $A$.
If $q$ is a root of unity, then $\cO_\bq(\KK^3)$ is a PI-algebra and so is $A$.
\end{example}

\begin{example}\label{exampleB}[The Hopf algebras $B(n, p_{0}, \ldots, p_{s}, q)$]
Let $n, p_{0}, \ldots, p_{s}$ be positive integers and let $q \in \KK^{\times}$ be such that
\begin{itemize}
\item[(a)] $s \geq 2$ and $1 < p_{0} < p_{1} < \ldots < p_{s}$;
\item[(b)] $p_{0} \mid n$ and $p_{1}, \ldots, p_{s}$ are relatively prime;
\item[(c)] $q$ is a primitive $l$th root of unity, where $l = (n/p_{0})p_{1}\ldots p_{s}$.
\end{itemize}
Let $m = p_{1}\ldots p_{s}$ and $m_{i} = m/p_{i}$. The Hopf algebra $B(n, p_{0}, \ldots, p_{s}, q)$ is the skew Laurent polynomial ring $A[x^{\pm}; \sigma]$ where $A$ is the subalgebra of the polynomial algebra $\KK[y]$ generated by the powers $y^{m_{i}}$ of the indeterminate $y$. The automorphism  $\sigma$ is the restriction of the $\KK$-algebra automorphism on $\KK[y]$ which sends $y$ to $qy$.  Since $q$ is a root of unity, $\KK[y][x^\pm; \sigma]$ is a PI-algebra, and so is $B(n,p_{0}, \ldots, p_{s}, q)$. Thus, by Theorem \ref{sufficient}, locally finite representations over $B(n,p_{0}, \ldots, p_{s}, q)$ are essentially closed.
\end{example}

\begin{example}\label{exampleC}[The Hopf Algebras $C(n)$]
Let $n$ be a positive integer and let $C(n) = \KK[y^{\pm}][x; (y^{n} - y)\frac{d}{dy}]$. There is a unique Hopf algebra structure on $C(n)$ such that $y$ is grouplike and $x$ is skew primitive, with $\Delta(x) = x \otimes y^{n-1} + 1 \otimes x$. The counit $\epsilon: C(n) \To \KK$ is such that $\epsilon(x) = 0$ and $\epsilon(y) = 1$. By Theorem~\ref{unmixed}, locally finite representations over the differential operator ring $C(n)$  are essentially closed.
\end{example}

\begin{cor} Let $\KK$ be algebraically closed of characteristic $0$, and let $H$ be an affine  Hopf $\KK$-algebra domain of Gelfand-Kirillov dimension $2$ satisfying $\operatorname{Ext}_H^1(\KK,\KK)\neq 0$. Then $\Loc{H}$ is essentially closed.
\end{cor}

\begin{proof}
Going through Goodearl and Zhang's classification result, we have seen that $H$  is either isomorphic to a group algebra, an universal enveloping algebra or an algebra of type $A(n,q)$, $B(n, p_{0}, \ldots, p_{s}, q)$ or $C(n)$. Furthermore, we have seen in Examples \ref{exampleA}, \ref{exampleB}, \ref{exampleC}  that $\Loc{H}$ is essentially closed in case that $H$ is belongs to one of  the latter three classes of algebras. If $H\simeq U(\g)$ with $\g$ a two-dimensional Lie algebra, then $\g$ is either Abelian or isomorphic to the the unique $2$-dimensional solvable, non-nilpotent Lie algebra (see \cite{ErdmannWildon}*{Theorem 3.1}). In both cases, $\Loc{U(\g)}$ is essentially closed by Donkin's result \cite{Donkin82}*{Proposition 2.2.2}. In the case $H \simeq \KK G$ with $G = \ZZ^2$ or $G=\ZZ \rtimes \ZZ$ then again $\Loc{H}$ is essentially closed again by Donkin's result as $G$ is polycyclic \cite{Donkin}*{Theorem 1.1.1}.
\end{proof}

\subsubsection{}
In \cite{Wang-Zhang-Zhuang}, Wang {\emph{et al.}} found a class of affine Hopf algebra domains $H$ over $\KK$ of Gelfand-Kirillov dimension $2$ with $\operatorname{Ext}_H^1(\KK,\KK)= 0$. The algebras they found satisfy a polynomial identity and hence satisfy the condition that locally finite representations are essentially closed. Wang {\emph{et al.}} conjectured that their family of Hopf algebras together with the Hopf algebras found by Goodearl and Zhang (see Theorem \ref{GZ}) are all affine Hopf algebra domains of Gelfand-Kirillov dimension $2$. If true, then any affine Hopf algebra domains of Gelfand-Kirillov dimension less or equal to $2$ would be essentially closed. Recall, that $H=U(\mathfrak{sl}_2)$ is an example of an affine Hopf algebra domain with Gelfand-Kirillov dimension three, whose locally finite representations are not closed under essential extensions.

\bigskip

{\bf Acknowledgments}

\medskip
The authors would like to thank the following persons: Ken Brown for having read and commented on an earlier version of this paper and for many valuable conversations and suggestions over the last years; Paula Carvalho and Allen Bell for comments on the strong second layer conditions and the Artin-Rees property; the referee for her/his careful reading of the manuscript, for suggestions that improved Theorem 2.9  and for having spotted a gap in an earlier version of section 4.1.3. This work was initiated when Can Hat\.{i}po\u{g}lu visited the Mathematics Department of the University of Porto in 2018. He wishes to thank the department for the hospitality and financial support.
The author Christian Lomp was partially supported by CMUP, which is financed by national funds through FCT – Funda\c{c}\~{a}o para a Ci\^{e}ncia e a Tecnologia, I.P., under the project with reference UIDB/00144/2020.


\end{document}